\documentclass[]{proc-l}

\usepackage{amssymb}
\usepackage{tikz}
\usepackage{comment}
\usetikzlibrary{arrows,automata,decorations.pathmorphing}

\newcommand{\Z}{\mathbb{Z}}
\newcommand{\N}{\mathbb{N}}

\newcommand{\B}{\mathcal{B}}

\newcommand{\Pa}{\mathcal{P}}
\newcommand{\INF}{{}^\infty}
\newcommand{\hole}{{\diamond}}

\newtheorem{theorem}{Theorem}
\newtheorem{lemma}{Lemma}
\newtheorem{proposition}{Proposition}
\newtheorem{corollary}{Corollary}

\theoremstyle{definition}
\newtheorem{definition}{Definition}
\newtheorem{problem}{Problem}
\newtheorem{example}{Example}

\theoremstyle{remark}
\newtheorem{claim}{Claim}

\begin{document}

\title{Quantifier Extensions of Multidimensional Sofic Shifts}


\author{Ilkka T\"orm\"a}
\address{TUCS -- Turku Centre for Computer Science \\ University of Turku, Finland}
\email{iatorm@utu.fi}
\thanks{Research supported by the Academy of Finland Grant 131558}


\date{\today}


\commby{Nimish Shah}

\begin{abstract}
We define a pair of simple combinatorial operations on subshifts, called existential and universal extensions, and study their basic properties. We prove that the existential extension of a sofic shift by another sofic shift is always sofic, and the same holds for the universal extension in one dimension. However, we also show by a construction that universal extensions of two-dimensional sofic shifts may not be sofic, even if the subshift we extend by is very simple.
\end{abstract}

\maketitle

\section{Introduction}

Subshifts of finite type and sofic shifts are, in some sense, the simplest objects studied in symbolic dynamics. Subshifts of finite type are symbolic systems whose structure is defined by local interactions, and sofic shifts are obtained from them by forgetting some of their structure. One-dimensional SFTs and sofic shifts are quite well-behaved and their theory is well understood, although some fundamental open problems still remain. For a good overview of the theory, see \cite{LiMa95}. In particular, since one-dimensional sofic shifts correspond to regular languages, most natural (and some unnatural) operations respect the property of being sofic.

The higher-dimensional case is much more complex, and the language of a two-dimensional SFT can already be uncomputable. In fact, even the emptiness problem of two-dimensional SFTs is undecidable \cite{Be66}. The class of two-dimensional sofic shifts is badly understood, in the sense that there are only a couple of general conditions for showing that a particular subshift is not sofic, and not many more specific ones \cite{De06,DuRoSh10,KaMa13,Pa13}. In this article, we study a pair of particular operations on subshifts, called the \emph{quantifier extensions}, which `extend' a given subshift by another subshift. They both respect the property of being a sofic shift in one dimension, and are inspired by the concept of \emph{multi-choice shift spaces}, as defined in \cite{LoMaPa13} and studied further in \cite{MePa11}. The main theorem shows that one of the operations does not take sofic shifts to sofic shifts in the two-dimensional setting, even when extending by very simple subshifts. This also solves an open problem presented in \cite{LoMaPa13}.

\section{Definitions}
\label{sec:Defs}

Let $S$ be a finite set, called the \emph{alphabet}, and fix a dimension $d \geq 1$. The set of finite \emph{words} over $S$ is denoted $S^*$, and the length of a word $w \in S^*$ by $|w|$. A \emph{$d$-dimensional pattern over $S$} is a function $P : D \to S$, where $D \subset \Z^d$ is the \emph{domain} of $P$, denoted $D(P)$. The set of finite $d$-dimensional patterns over $S$ is denoted $\Pa_d(S)$. A full pattern $x : \Z^d \to S$ is called a \emph{configuration}. A pattern $P : D \to S$ \emph{occurs} in another pattern $Q : E \to S$, denoted $P \sqsubset Q$, if $\vec v + \vec w \in E$ and $Q_{\vec v + \vec w} = P_{\vec v}$ hold for some $\vec w \in \Z^d$ and all $\vec v \in D$.

Each set of finite patterns $F \subset \Pa_d(S)$ defines a \emph{subshift} as the set of configurations $X = \{ x \in S^{\Z^d} \;|\; \forall P \in F : P \not\sqsubset x \}$. We say $F$ is a \emph{set of forbidden patterns} for $X$. If $F$ is finite, $X$ is a \emph{subshift of finite type}, SFT for short, and if every pattern of $F$ has domain $\{\vec 0, \vec e_i\}$ for some $i \in [1,d]$, where $\{\vec e_1, \ldots, \vec e_d\}$ is the natural generating set of $\Z^d$, $X$ is a \emph{tiling system}. The set $\{ P \in \Pa_d(S) \;|\; \exists x \in X : P \sqsubset x \}$ is called the \emph{language} of $X$, and is denoted $\B(X)$. For $D \subset \Z^d$, we also denote $\B_D(X) = \B(X) \cap S^D$, and if $d = 1$ and $n \in \N$, we denote $\B_n(X) = \B(X) \cap S^n$. A subshift $X \subset S^{\Z^d}$ is \emph{strongly irreducible (with constant $M > 0$)} if for any two domains $D, D' \subset \Z^d$ such that $\min \{ \| \vec n - \vec m \| \;|\; \vec n \in D, \vec m \in D' \} \geq M$ where $\| \cdot \|$ denotes the maximum norm, and any two patterns $p \in \B_D(X)$ and $p' \in \B_{D'}(X)$, there exists $x \in X$ with $x_D = p$ and $x_{D'} = p'$.

A function $\phi : X \to Y$ between two subshifts $X \subset S^{\Z^d}$ and $Y \subset T^{\Z^d}$ is called a \emph{block map} if there exists a \emph{local function} $\Phi : \B_D(X) \to T$, where $D \subset \Z^d$ is finite, such that $\phi(x)_{\vec v} = \Phi(x_{D + \vec v})$ for all $x \in X$ and $\vec v \in \Z^d$. If $D = \{\vec 0\}$, then $\phi$ is a \emph{symbol map}. The image of an SFT under a block map is called a \emph{sofic shift}. All sofic shifts also occur as images of tiling systems under symbol maps.

We recall the \emph{arithmetical hierarchy}, which classifies certain subsets of $\N$ according to their complexity. A first-order arithmetical formula with bounded quantifiers is classified as $\Pi^0_0$ and $\Sigma^0_0$. If a formula $\phi$ with free variable $k$ is $\Pi^0_n$, then $\exists k : \phi$ is $\Sigma^0_{n+1}$. Similarly, if $\phi$ is $\Sigma^0_n$, then $\forall k : \phi$ is $\Pi^0_{n+1}$. Every first-order arithmetical formula is equivalent to a $\Pi^0_n$ or $\Sigma^0_n$ one, for some $n \in \N$. If $\phi$ has a unique free variable $k$, then the set $N_\phi = \{ k \in \N \;|\; \phi(k) \}$ is given the same classification as $\phi$. Note that $N_\phi$ is $\Pi^0_n$ if and only if $\N \setminus N_\phi$ is $\Sigma^0_n$. A set $N \subset \N$ is $C$-\emph{hard} for a class $C$, if for every $M \in C$ there is a computable function $f : \N \to \N$ with $M = f^{-1}(N)$, and $C$-\emph{complete} if also $N \in C$. The class $\Sigma^0_1$ ($\Pi^0_1$) contains exactly the (co-)recursively enumerable sets. When classifying subsets of other countable sets than $\N$, for example $\Pa_d(S)$, we implicitly choose a computable bijection to $\N$. A subshift is given the same classification as its language.


\section{The Quantifier Extensions}

We begin by defining our objects of interest, the quantifier extension subshifts.

\begin{definition}
Let $X, Y \subset S^{\Z^d}$ be $d$-dimensional subshifts, let $\hole \notin S$ be a new symbol and denote $\hat S = S \cup \{\hole\}$. For patterns $P \in \Pa_d(\hat S)$ and $Q \in S^{D(P)}$, denote by $P^{(Q)} \in S^{D(P)}$ the pattern with
\[ P^{(Q)}_{\vec v} = \left\{ \begin{array}{ll}
	Q_{\vec v}, & \mbox{if~} P_{\vec v} = \hole \\
	P_{\vec v}, & \mbox{otherwise.}
\end{array} \right. \]
We define two \emph{quantifier extension subshifts} for $X$ and $Y$, the \emph{universal extension} $A(X,Y) \subset \hat S^{\Z^d}$ and the \emph{existential extension} $E(X,Y) \subset \hat S^{\Z^d}$, with the respective sets of forbidden patterns
\begin{align*}
\{ P \in \Pa_d(S) \;|\; \exists Q \in \B_{D(P)}(Y) : P^{(Q)} \notin \B(X) \} & \mbox{~and} \\
\{ P \in \Pa_d(S) \;|\; \forall Q \in \B_{D(P)}(Y) : P^{(Q)} \notin \B(X) \}
\end{align*}
\end{definition}

The symbol $\hole$ represents a `hole' in a configuration $x \in \hat S^{\Z^d}$ that can be filled with the contents of another configuration $y \in Y$. The extensions quantify over all such $y$ to decide whether $x$ is valid. Our perspective in this paper is to study which properties the extensions respect. 
We mainly focus on the universal extension, as the question of soficness is particularly interesting for it. First, we show that the universal extension respects the property of being an SFT in the following sense.

\begin{proposition}
\label{prop:SFTs}
Let $X \subset S^{\Z^d}$ be an SFT and $Y \subset S^{\Z^d}$ any subshift. Then the universal extension $A(X,Y)$ is an SFT.
\end{proposition}

\begin{proof}
Let $F \subset \Pa_d(S)$ be a finite set of forbidden patterns for $X$. We claim that $\hat F = \{ P \in \Pa_d(\hat S) \;|\; \exists Q \in \B_{D(P)}(Y) : P^{(Q)} \in F \}$ is a set of forbidden patterns for $A(X,Y)$. Let $x \in \hat S^{\Z^d}$ be arbitrary. If $x \in A(X,Y)$, it is clear that no patterns of $\hat F$ occur in $x$. Conversely, if $x \notin A(X,Y)$, then there exists a configuration $y \in Y$ such that $x^{(y)} \notin X$, or in other words, there exists $D \subset \Z^2$ such that $x^{(y)}_D \in F$. This implies that the pattern $P = x_D$ is in $\hat F$, since the pattern $Q = y_D \in \B_D(Y)$ satisfies $P^{(Q)} \in F$. Since $\hat F$ is finite, $A(X,Y)$ is an SFT, as claimed.
\end{proof}

\begin{example}
\label{ex:1DSFT}
The above result does not hold for the existential extension, even if $Y = \{0,1\}^\Z$ and $X \subset \{0,1\}^\Z$ is mixing. Namely, let $X$ be defined by the single forbidden pattern $0100$, and consider the configurations $x = \INF 0 1 (0 \hole) \INF$ and $x' = \INF (0 \hole) 0 \INF$. We have $x \in E(X,Y)$ by substituting $1$ to every $\hole$, and $x' \in E(X,Y)$ by substituting $0$. It is easy to see that these are the only possible substitutions, and thus $\INF 0 1 (0 \hole)^n 0 \INF \notin E(X,Y)$ for all $n \geq 1$. Thus $E(X,Y)$ is not an SFT.
\end{example}

For soficness, we can say the following.

\begin{proposition}
\label{prop:Soficness}
Let $X, Y \subset S^{\Z^d}$ be sofic shifts. Then the existential extension $E(X,Y)$ is sofic. If $d = 1$, then the universal extension $A(X,Y)$ is also sofic.
\end{proposition}

\begin{proof}
Let $\phi : X' \to X$ and $\psi : Y' \to Y$ be surjective block maps, where $X'$ and $Y'$ are SFTs, and denote $Z = \{\#,\hole\}^{\Z^d}$. Then $E(X,Y)$ is obtained from the SFT
\[ \{ (x,y,z) \;|\; \forall \vec v \in \Z^d: z_{\vec v} = \hole \implies x_{\vec v} = y_{\vec v} \} \subset X' \times Y' \times Z, \]
by applying the block map $\xi$ defined by
\[ \xi(x,y,z)_{\vec 0} = \left\{ \begin{array}{ll}
	\hole, & \mbox{if~} z_{\vec 0} = \hole \\
	\phi(x)_{\vec 0}, & \mbox{otherwise.}
\end{array} \right. \]

In the case $d = 1$, a subshift is sofic if and only if its language is regular, if and only if it can be defined by a regular forbidden set \cite[Chapter 3]{LiMa95}. The sets $L = \Pa_1(S) \setminus \B(X)$ and $\B(Y)$ are thus regular. Now, $A(X,Y)$ is defined by the set
\[ \bigcup_{n \in \N} \{ w \in \hat S^n \;|\; \exists v \in \B_n(Y) : w^{(v)} \in L \} \]
of forbidden words, which is clearly regular. Thus the extension is sofic.
\end{proof}

In higher dimensions, automata theory is replaced by computability theory. In particular, languages of multidimensional sofic shifts are co-recursively enumerable.

\begin{lemma}
\label{lem:Languages}
The language of every multidimensional sofic shift is $\Pi^0_1$.
\end{lemma}

\begin{proof}
Let $X \subset S^{\Z^d}$ be a sofic shift, let $F \subset \Pa_d(R)$ be a finite set of forbidden patterns defining an SFT $Y \subset R^{\Z^d}$, and let $\phi : Y \to X$ be a surjective symbol map. For a hypercube pattern $P \in S^{[0,n-1]^d}$, we have $P \in \B(X)$ if and only if
\[ \forall m \in \N: \exists Q \in R^{[-m,n+m-1]^d}: \phi(Q)_{[0,n-1]^d} = P \wedge (\forall T \in F : T \not\sqsubset Q). \]
This formula is $\Pi^0_1$ by form, and so the language of $X$ is $\Pi^0_1$.
\end{proof}

The latter statement of Proposition~\ref{prop:Soficness} translates to the following result.

\begin{lemma}
\label{lem:Computability}
Let $X, Y \subset S^{\Z^d}$ be $\Pi^0_1$ subshifts. Then the universal extension $A(X,Y)$ is $\Pi^0_2$. If $Y$ is also $\Sigma^0_1$, then $A(X,Y)$ is $\Pi^0_1$.
\end{lemma}

\begin{proof}
Let $P \in \hat S^{n \times n}$ be arbitrary. We have $P \notin \B(A(X,Y))$ if and only if
\[ \exists Q \in S^{n \times n}: Q \in \B(Y) \wedge P^{(Q)} \notin \B(X). \]
Since the languages of $X$ and $Y$ are $\Pi^0_1$, the proposition $Q \in \B(Y)$ is $\Pi^0_1$, while $P^{(Q)} \notin \B(X)$ is $\Sigma^0_1$. Thus the complement of $\B(A(X,Y))$ is a $\Sigma^0_2$ language, implying that $\B(A(X,Y))$ is $\Pi^0_2$. The latter claim follows similarly.
\end{proof}

This bound is sharp, and we use the following powerful result to prove it.

\begin{theorem}[\cite{DuRoSh10,AuSa13}]
\label{thm:Pi01}
Let $X \in S^\Z$ be a $\Pi^0_1$ subshift. Then the vertically periodic two-dimensional subshift $\{ y \in S^{\Z^2} \;|\; \exists x \in X: \forall (i,j) \in \Z^2: y_{(i,j)} = x_i \}$ is sofic.
\end{theorem}

\begin{proposition}
\label{prop:CompuNonsofic}
There exist countable sofic shifts $X, Y \subset S^{\Z^2}$ such that the language of the universal extension $A(X,Y)$ is $\Pi^0_2$-complete.
\end{proposition}

\begin{proof}
Because of Lemma~\ref{lem:Computability} and Theorem~\ref{thm:Pi01}, it suffices to construct one-dimensional countable $\Pi^0_1$ subshifts $X, Y \subset S^\Z$ such that $A(X,Y)$ is $\Pi^0_2$-hard.

Define $S = \{ 0, \ldots, 6 \}$, and define $X$ by the set of forbidden patterns
\[ \{ 4 6 \} \cup\{ i j \;|\; i, j \in S, i > j \} \cup \{ 0 1^a 2^{b+1} 3^c 4 \;|\; a, b, c \in \N, a = c \}. \]
Since this set is recursive, $X$ is a $\Pi^0_1$ subshift, and easily seen to be countable. Next, let $\Phi$ be an arithmetical formula with bounded quantifiers such that the set
\[ N = \{ k \in \N \;|\; \forall m \in \N: \exists n \in \N: \Phi(k,m,n) \} \]
is $\Pi^0_2$-hard. Define $Y$ by the set
\[ \{ 0, 1, 4 6 \} \cup \{ i j \;|\; i, j \in S, i > j \} \cup \{ 2 3^a 4^b 5 \;|\; a, b \in \N, \exists n \in \N : \neg \Phi(a, b, n) \} \]
of forbidden patterns. Since this set is $\Sigma^0_1$ by form, $Y$ is a countable $\Pi^0_1$ subshift.

Define the function $w : \N \to L = \{0 1^k 2 \hole \;|\; k \in \N\}$ by $w(k) = 0 1^k 2 \hole$. We note that for all $s \in S$, there exists $t \in \B_1(Y)$ such that $t s$ is forbidden in $X$, which implies that $\hole s$ is forbidden in $A(X,Y)$. Thus $w(k)$ occurs in $A(X,Y)$ if and only if the infinite tail $w(k) \hole \INF$ does. By the definition of $X$ and $Y$, this is the case if and only if $2 3^k 4^m 5 \notin \B(Y)$ for all $m \in \N$. But this is equivalent to $k \in N$, which means that $N = w^{-1}(L \cap \B(A(X,Y)))$, and thus the language of $A(X,Y)$ is $\Pi^0_2$-hard.
\end{proof}

As a corollary of Lemma~\ref{lem:Languages} and the above proposition, we obtain the following counterpart of Proposition~\ref{prop:Soficness}.

\begin{corollary}
There exist countable sofic shifts $X, Y \subset S^{\Z^2}$ such that the universal extension $A(X,Y)$ is not sofic.
\end{corollary}

While this result is interesting in itself, the proof is not very satisfying, since the subshift $Y$ that we extend by is computationally complex, and we use the simpler structure of $X$ only to check a universally quantified property of $\B(Y)$. 
However, if we restrict $Y$ 
to be a recursive subshift (both $\Pi^0_1$ and $\Sigma^0_1$), Lemma~\ref{lem:Computability} shows that the language of $A(X,Y)$ is $\Pi^0_1$, so a recursion theoretic proof will no longer work.

In the next section, we 
find pairs of computationally simple sofic shifts $X$ and $Y$ such that $A(X,Y)$ is not sofic. The following special case was presented as an open problem by professor Brian Marcus at the 2013 PIMS/EQINOCS Automata Theory and Symbolic Dynamics Workshop: Is the extension $A(X,\{0,1\}^{\Z^2})$ a sofic shift for every sofic $X \subset \{0,1\}^{\Z^2}$?
In \cite{LoMaPa13}, it was asked whether the \emph{multi-choice shift space} (see the cited article for the definition) associated to a two-dimensional sofic shift is necessarily sofic, and the above problem is a restatement of this question in the binary case. Theorem~\ref{thm:MainThm} in particular shows that the answer is negative.

\section{Main Theorem}

In this section, we classify those two-dimensional subshifts that only yield sofic extensions. Before stating the result, we give the following definition.

\begin{definition}
A sofic shift $X \subset S^{\Z^d}$ is \emph{countably covered} if it is the image of a countable SFT via a block map.
\end{definition}

This notion is not standard in the literature. Of course, all countably covered sofic shifts are countable. Now, our main result is the following.

\begin{theorem}
\label{thm:MainThm}
Let $Y \subset S^{\Z^2}$ be a subshift. The following are equivalent:
\begin{enumerate}
\item $Y$ is finite.
\item $A(X,Y)$ is sofic for all sofic shifts $X \subset R^{\Z^2}$ over all alphabets $R$.
\item $A(X,Y)$ is sofic for all countably covered sofic shifts $X \subset S^{\Z^2}$.
\end{enumerate}
\end{theorem}

We continue with a series of lemmas. Most of them are not needed in the special case of binary full shifts, but we use them because Theorem~\ref{thm:MainThm} is much more general.

\begin{lemma}
\label{lem:CountableSofics}
Let $X \subset S^{\Z^2}$ be a countably covered sofic shift, and let $Y \subset (S \times R)^{\Z^2}$ be an SFT. If $Y \cap (X \times R^{\Z^2})$ is countable, then it is a countably covered sofic shift.
\end{lemma}

\begin{proof}
Let $X = f(Z)$, where $Z \subset T^{\Z^2}$ is a countable SFT and $f : Z \to X$ a block map with neighborhood $N \subset \Z^2$. Let $F \subset T^D$ and $G \subset (S \times R)^D$ be sets of forbidden patterns for $Z$ and $Y$, where $D \subset \Z^2$ is finite. Then the set $F \times R^D \cup \{ (P,Q) \in (T \times R)^{D+N} \;|\; (f(P), Q|_D) \in G \}$ defines the SFT $Y' = \{ (z, z') \in Z \times R^{\Z^2} \;|\; (f(z), z') \in Y \}$, which is countable since $Y$ is. The image of $Y'$ under the block map $(z, z') \mapsto (f(z), z')$ is exactly $Y$, and the claim follows.
\end{proof}

The following is one of the few known methods for showing a multidimensional subshift to be nonsofic. A proof of it has appeared at least in \cite{KaMa13}, but the technique is much older, presumably originating from the theory of picture languages. In the proof, the notation $\partial D$ for a domain $D \subset \Z^2$ stands for the set
\[ \{ (a,b) \in \Z^2 \;|\; (a,b) \notin D, \{(a+1,b),(a-1,b),(a,b+1),(a,b-1)\} \cap D \neq \emptyset \} \subset \Z^2. \]

\begin{lemma}
\label{lem:Contexts}
Let $X \subset S^{\Z^2}$ be a sofic shift. Then there exists $C > 0$ such that for all $n \geq 1$ and all sets $\Lambda \subset X$ of size at least $C^n + 1$, there exist $x \neq y \in \Lambda$ such that $c(x,y,n) \in X$, where $c(x,y,n) \in S^{\Z^2}$ is defined by $c(x,y,n)_{[0,n-1]^2} = x_{[0,n-1]^2}$ and $c(x,y,n)_{\Z^2 \setminus [0,n-1]^2} = y_{\Z^2 \setminus [0,n-1]^2}$.
\end{lemma}

\begin{proof}
Since $X$ is sofic, there exist a tiling system $Y \subset T^{\Z^2}$ and a surjective symbol map $\phi : Y \to X$. We claim that $C = |T|^4$ suffices. Let $n \geq 1$, and for all $x \in X$, choose a preimage $\tilde x \in \phi^{-1}(x)$. Let $x, y \in X$. If we have $\tilde x_{\partial [0,n-1]^2} = \tilde y_{\partial [0,n-1]^2}$, then the configuration $c(\tilde x, \tilde y, n) \in T^{\Z^2}$ is in $Y$, so $\phi(c(\tilde x, \tilde y, n)) = c(x, y, n) \in X$. Since the number of $\tilde x_{\partial [0,n-1]^2}$ for $x \in X$ is at most $C^n = |T|^{4n}$, the claim follows.
\end{proof}

The proof of Theorem~\ref{thm:MainThm} requires the use of a computational device, and the following nonstandard version of the classical counter machine suits our needs.

\begin{definition}
A \emph{counter machine with string input} (CMS for short) is a tuple $M = (k,k',S,\Sigma,\delta,q_0,q_a,q_r)$, where $k \in \N$ is the number of \emph{counters}, $k' < k$ the number of \emph{output counters}, $\Sigma$ is a finite \emph{state set}, $S$ the finite \emph{input alphabet}, $q_0, q_a, q_r \in \Sigma$ the \emph{initial, accepting and rejecting states} and
\[ \delta \subset (\Sigma \times [1,k] \times \{Z,P\} \times \Sigma) \cup (\Sigma \times [1,k] \times \{-1,0,1\} \times \Sigma) \cup (\Sigma \times \bar S \times \Sigma) \]
the \emph{transition relation}, where $\bar S = S \cup \{\#\}$ for a new symbol $\# \notin S$. An \emph{instantaneous description (ID) of $M$} is an element $(q,w,n_1,\ldots,n_k)$ of $\Sigma \times S^* \times \N^k$, with the interpretation that $M$ is in state $q$ with input word $w$ and counter values $n_1, \ldots, n_k$.

The CMS $M$ operates in possibly nondeterministic steps as follows. If we have $(p, i, Z, q) \in \delta$ ($(p, i, P, q) \in \delta$), then from any ID $(p,w,n_1,\ldots,n_k)$ with $n_i = 0$ ($n_i > 0$, respectively), $M$ may move to the ID $(q,w,n_0,\ldots,n_k)$. If $(p, i, r, q) \in \delta$ with $r \in \{-1,0,1\}$, then from any ID $(p,w,n_1,\ldots,n_i,\ldots,n_k)$, $M$ may move to $(q,w,n_1,\ldots,n_i+r,\ldots,n_k)$. Finally, if $(p, s, q) \in \delta$ for $s \in S$ ($(p, \#, q) \in \delta$), then then from any ID $(p,w,n_1,\ldots,n_k)$ with $w_{n_k} = s$ ($|w| \leq n_k$, respectively), $M$ may move to $(q,w,n_1,\ldots,n_k)$. We assume that $M$ never decrements a counter below $0$.

The CMS is initialized from an ID $(q_0, w, 0, \ldots, 0)$ for $w \in S^*$, and halts when it reaches the states $q_a$ or $q_r$. The tuple $(n_1,\ldots,n_{k'}) \in \N^{k'}$ of the first $k'$ counter values in the accepting state $q_a$ is the \emph{output} of $M$ on $w$, and is denoted $M(w)$. Note that $n_1$ is the value of the first counter, not the greatest counter value. If $M$ reaches the rejecting state $q_r$ or never halts, no output is generated.
\end{definition}

In one step, a CMS may increment or decrement one of its counters, check whether a counter is zero, or check the input symbol under the last counter, if one exists. When it halts in the accepting state, the counter values of the final ID are considered as outputs. Thus a CMS can be interpreted as a partial function from $S^*$ to $\N^{k'}$. The classical reference for counter machines is \cite{Mi67}, although CMSs are not defined there. Conventional counter machines are computationally universal, and it is not hard to see that the same holds for these devices.

\begin{lemma}
Every recursive partial function $f : S^* \to \N^{k'}$ can be realized as a CMS.
\end{lemma}

We also need the following result from symbolic dynamics. It is slightly stronger than the version in \cite{LiMa95}, but the missing details can be extracted from its proof.

\begin{lemma}[Marker Lemma]
\label{lem:Marker}
Let $S$ be a finite alphabet. For all $n \in \N$ there exists a block map $f_n : S^\Z \to \{0,1\}^\Z$, called the \emph{$n$-marker map}, such that:
\begin{enumerate}
\item The radius of $f_n$ is at most $r_n = n^{|S|^{2n+1}}$.
\item For $k < n-1$, the word $1 0^k 1$ does not occur in any configuration of $f_n(S^\Z)$.
\item If $x \in S^\Z$ is such that $f_n(x)_{[-n+1,n-1]} = 0^{2n-1}$, then $x_{[-n,n]}$ is periodic with period less than $n$.
\item The function $(n, w) \mapsto F_n(w)$, where $w \in S^{r_n}$ and $F_n$ is the local function of $f_n$, is recursive.
\end{enumerate}
\end{lemma}

\begin{proof}
Lemma~10.1.8 of \cite{LiMa95} states the existence of a clopen set $C \subset S^\Z$ such that the shifted sets $\sigma^i(C)$ are disjoint for all $i \in [0, n-1]$, and if $\sigma^i(x) \notin C$ for all $i \in [-n+1, n-1]$, then $x_{[-n,n]}$ is periodic with period less than $n$. We define $f_n(x)_i = 1$ if and only if $\sigma^i(x) \in C$, and then $f_n : S^\Z \to \{0,1\}^\Z$ is a block map satisfying (2) and (3). The other claims follow from the construction of $C$ in \cite{LiMa95}.
\end{proof}

Finally, we have a general construction of grid-like countably covered sofic shifts.

\begin{lemma}
\label{lem:GridsExist}
For all $m, n \in \N \cup \{\infty\}$, there exists a countably covered sofic shift $X = X_{G(m,n)}$ over the alphabet $\{\#\} \cup \{0,1\}^2$ such that:
\begin{itemize}
\item For all $x \in X$, the set $D(x) = \{\vec v \in \Z^2 \;|\; x_{\vec v} \neq \# \}$ is a (possibly infinite) rectangle.
\item For all $k \geq 2$, $1 \leq a \leq \binom{m+k}{k}$ and $1 \leq b \leq \binom{n+k}{k}$, there is a configuration $x = x^{k,a,b} \in X$ such that $D(x) = [0, a k] \times [0, b k]$, and for all $\vec v = (i,j) \in D(x)$ we have $\pi_1(x_{\vec v}) = 1$ ($\pi_2(x_{\vec v}) = 1$) if and only if $i \equiv 0 \bmod k$ ($j \equiv 0 \bmod k$, respectively), where $\pi_1, \pi_2 : \{0,1\}^2 \to \{0,1\}$ are the projections to the first and second components.
\item If $m = \infty$ ($n = \infty$), then every horizontal (vertical) line of $1$'s in the second (first) layer of $X$ is infinite to the right (upwards), and otherwise, every configuration of $X$ contains only finitely many vertical (horizontal) lines of $1$'s in its first (second) layer.
\end{itemize}
\end{lemma}

\begin{proof}
We construct a countable SFT $Y \subset A^{\Z^2}$ and a symbol map $\pi : A \to \{0,1\}$ such that $\pi(Y) = X$. The alphabet $A$ is the set of tiles in Figure~\ref{fig:GridAlph}, where the labels $C_\ell$ range over $[0,\ell]$ if $\ell \in \N$, and $\{\infty\}$ if $\ell = \infty$. Note that some tiles are forbidden if $m$ or $n$ is infinite. Every $2 \times 2$ pattern where the lines or colors of some tiles do not match (including the diagonal lines) is forbidden in $Y$. Then the regions colored by $L$, $R$, $B$, and $T$ in a configuration $y \in Y$, if nonempty, form a left half plane, a right half plane, a downward infinite rectangle, and an upward infinite rectangle, respectively. The rectangular area not contained in them is called the \emph{grid} of $y$. It is divided into rectangles by the \emph{grid lines} (the thick lines in Figure~\ref{fig:GridAlph}), which stretch from one end of the grid to the other. These rectangles are actually squares, all of the same size, because of the diagonal lines. We also forbid every $2 \times 2$ pattern containing such a square, so that every square contains at least one \emph{interior tile}, shown in the fifth column of the figure. If $m$ ($n$) is infinite, then so is the width (height) of every grid, as it cannot have a right (top, respectively) border.

The labels $C_m$ of the interior tiles have the following rules if $m \neq \infty$. Inside a square of the grid, they must be horizontally constant and downward increasing (all patterns $c d$ for $c \neq d$ and $\begin{smallmatrix} c \\ d \end{smallmatrix}$ for $c > d$ are forbidden). On the border of two horizontally adjacent squares, all $3 \times 2$ patterns except
\[
\left[ \begin{smallmatrix} {-} & {+} & {-} \\ e & | & e \end{smallmatrix} \right],
\left[ \begin{smallmatrix} {-} & {+} & {-} \\ e & | & e+1 \end{smallmatrix} \right],
\left[ \begin{smallmatrix} e & | & e \\ f & | & f \end{smallmatrix} \right],
\left[ \begin{smallmatrix} e & | & e \\ f & | & f+1 \end{smallmatrix} \right],
\left[ \begin{smallmatrix} e & | & e+1 \\ m & | & e+1 \end{smallmatrix} \right],
\left[ \begin{smallmatrix} m & | & c \\ m & | & c \end{smallmatrix} \right],
\left[ \begin{smallmatrix} e & | & e+1 \\ {-} & {+} & {-} \end{smallmatrix} \right],
\left[ \begin{smallmatrix} m & | & c \\ {-} & {+} & {-} \end{smallmatrix} \right]
\]
for $c, d \in [0,m]$ and $e, f \in [0,m-1]$ are forbidden (the symbols ${-}$, ${|}$ and ${+}$ represent horizontal, vertical and crossing grid lines, including T-junctions). Denote by $V_m^k$ the set of length-$k$ downward increasing column vectors over $[0,m]$. The above rules imply that for all horizontally adjacent grid squares with $k \times k$ interior tiles, the column vector formed by the bottom-right labels of the interior tiles of the right square is the lexicographical successor of that of the left square with respect to the set $V_m^k$. Thus the width of any grid containing such a square is at most $|V_m^k| = \binom{m+k}{k}$ squares. To ensure countability, we also require that the labels next to the left border of a grid are all $0$. We introduce analogous rules for the top-left labels, but transposed, so that the height of the grid is at most $\binom{n+k}{k}$ squares. This concludes the definition of $Y$. We set $\pi(t) \neq \#$ if and only if the tile $t$ contains a gray region, and then $\pi_1(\pi(t)) = 1$ ($\pi_2(\pi(t)) = 1$) if and only if $t$ contains a vertical (horizontal) grid line. The three conditions for $X$ follow easily.

We show that $Y$ is countable, so let $y \in Y$. If a finite grid square occurs in $y$, then there are countably many choices for the position of the grid, which uniquely determines its contents (because of the restrictions on the column vectors introduced above) and the rest of $y$. If $y$ contains no grid tiles, then it consists of the $L$, $R$, $B$ and $T$-tiles, for which we have countably many choices. In the case of infinite squares, since the labels of the interior tiles are decreasing in one direction and constant in the other, our choices are again restricted to a countable set.
\end{proof}

\begin{figure}[htp]
\begin{center}
\begin{tikzpicture}[scale=.8]


  \draw (0,0) rectangle (1,1);
  \node () at (.5,.5) {$L$};

  \begin{scope}[shift={(1.5,0)}]
  \fill[black!30] (.5,.5) -- (1,1) -- (1,.5);
  \fill[black!15] (.5,.5) -- (1,1) -- (.5,1);
  \draw (0,0) rectangle (1,1);
  \draw[densely dotted] (.5,0) -- (.5,.5);
  \draw[very thick] (.5,1) -- (.5,.5) -- (1,.5);
  \draw[dashed] (.5,.5) -- (1,1);
  \node () at (.25,.5) {$L$};
  \node () at (.75,.25) {$B$};
  \end{scope}

  \begin{scope}[shift={(3,0)}]
  \fill[black!30] (0,.5) rectangle (1,1);
  \draw (0,0) rectangle (1,1);
  \draw[very thick] (0,.5) -- (1,.5);
  \node () at (.5,.25) {$B$};
  \end{scope}

  \begin{scope}[shift={(4.5,0)}]
  \fill[black!30] (0,.5) rectangle (.5,1);
  \fill[black!30] (.5,.5) -- (1,1) -- (1,.5);
  \fill[black!15] (.5,.5) -- (1,1) -- (.5,1);
  \draw (0,0) rectangle (1,1);
  \draw[very thick] (0,.5) -- (1,.5);
  \draw[very thick] (.5,.5) -- (.5,1);
  \draw[dashed] (.5,.5) -- (1,1);
  \node () at (.5,.25) {$B$};
  \end{scope}

  \begin{scope}[shift={(6,0)}]
  \draw[fill=black!30] (0,0) rectangle (1,1);
  \node () at (1/3,3/4) {$C_m$};
  \node () at (2/3,1/4) {$C_n$};
  \end{scope}

  \begin{scope}[shift={(7.5,0)}]
  \fill[black!30] (0,.5) rectangle (.5,1);
  \draw (0,0) rectangle (1,1);
  \draw[densely dotted] (.5,0) -- (.5,.5);
  \draw[very thick] (0,.5) -- (.5,.5) -- (.5,1);
  \node () at (.75,.5) {$R$};
  \node () at (.25,.25) {$B$};
  \end{scope}

  \begin{scope}[shift={(9,0)}]
  \draw (0,0) rectangle (1,1);
  \node () at (.5,.5) {$R$};
  \end{scope}

  \begin{scope}[shift={(10.5,0)}]
  \draw (0,0) rectangle (1,1);
  \node () at (.5,.5) {$B$};
  \end{scope}


  \begin{scope}[shift={(0,1.5)}]
  \fill[black!15] (.5,0) rectangle (1,1);
  \draw (0,0) rectangle (1,1);
  \draw[very thick] (.5,0) -- (.5,1);
  \node () at (.25,.5) {$L$};
  \end{scope}

  \begin{scope}[shift={(1.5,1.5)}]
  \fill[black!15] (.5,.5) rectangle (1,0);
  \fill[black!15] (.5,.5) -- (1,1) -- (.5,1);
  \fill[black!30] (.5,.5) -- (1,1) -- (1,.5);
  \draw (0,0) rectangle (1,1);
  \draw[very thick] (.5,0) -- (.5,1);
  \draw[very thick] (.5,.5) -- (1,.5);
  \draw[dashed] (.5,.5) -- (1,1);
  \node () at (.25,.5) {$L$};
  \end{scope}

  \begin{scope}[shift={(3,1.5)}]
  \fill[black!15] (0,0) rectangle (1,.5);
  \fill[black!30] (0,.5) rectangle (1,1);
  \draw (0,0) rectangle (1,1);
  \draw[very thick] (0,.5) -- (1,.5);
  \end{scope}

  \begin{scope}[shift={(4.5,1.5)}]
  \fill[black!30] (0,.5) rectangle (.5,1);
  \fill[black!15] (.5,0) rectangle (1,.5);
  \fill[black!30] (.5,.5) -- (1,1) -- (1,.5);
  \fill[black!15] (.5,.5) -- (1,1) -- (.5,1);
  \fill[black!30] (0,0) -- (.5,.5) -- (.5,0);
  \fill[black!15] (0,0) -- (.5,.5) -- (0,.5);
  \draw (0,0) rectangle (1,1);
  \draw[very thick] (0,.5) -- (1,.5);
  \draw[very thick] (.5,0) -- (.5,1);
  \draw[dashed] (0,0) -- (1,1);
  \end{scope}

  \begin{scope}[shift={(6,1.5)}]
  \fill[black!30] (0,0) -- (1,1) -- (1,0);
  \fill[black!15] (0,0) -- (1,1) -- (0,1);
  \node () at (1/3,3/4) {$C_m$};
  \node () at (2/3,1/4) {$C_n$};
  \draw[dashed] (0,0) -- (1,1);
  \draw (0,0) rectangle (1,1);
  \end{scope}

  \begin{scope}[shift={(7.5,1.5)}]
  \fill[black!30] (0,.5) rectangle (.5,1);
  \fill[black!30] (0,0) -- (.5,.5) -- (.5,0);
  \fill[black!15] (0,0) -- (.5,.5) -- (0,.5);
  \draw (0,0) rectangle (1,1);
  \draw[dashed] (0,0) -- (.5,.5);
  \draw[very thick] (.5,0) -- (.5,1);
  \draw[very thick] (0,.5) -- (.5,.5);
  \node () at (.75,.5) {$R$};
  \end{scope}

  \begin{scope}[shift={(9,1.5)}]
  \fill[black!30] (0,0) rectangle (.5,1);
  \draw (0,0) rectangle (1,1);
  \draw[very thick] (.5,0) -- (.5,1);
  \node () at (.75,.5) {$R$};
  \end{scope}

  \begin{scope}[shift={(10.5,1.5)}]
  \draw (0,0) rectangle (1,1);
  \node () at (.5,.5) {$T$};
  \end{scope}


  \begin{scope}[shift={(0,3)}]
  \draw (0,0) rectangle (1,1);
  \draw[densely dotted] (.5,0) -- (.5,1);
  \node () at (.25,.5) {$L$};
  \node () at (.75,.5) {$T$};
  \end{scope}

  \begin{scope}[shift={(1.5,3)}]
  \fill[black!15] (.5,.5) rectangle (1,0);
  \draw (0,0) rectangle (1,1);
  \draw[very thick] (.5,0) -- (.5,.5) -- (1,.5);
  \draw[densely dotted] (.5,.5) -- (.5,1);
  \node () at (.25,.5) {$L$};
  \node () at (.75,.75) {$T$};
  \end{scope}

  \begin{scope}[shift={(3,3)}]
  \fill[black!15] (0,.5) rectangle (1,0);
  \draw (0,0) rectangle (1,1);
  \draw[very thick] (0,.5) -- (1,.5);
  \node () at (.5,.75) {$T$};
  \end{scope}

  \begin{scope}[shift={(4.5,3)}]
  \fill[black!15] (.5,.5) rectangle (1,0);
  \fill[black!30] (.5,.5) -- (0,0) -- (.5,0);
  \fill[black!15] (.5,.5) -- (0,0) -- (0,.5);
  \draw (0,0) rectangle (1,1);
  \draw[very thick] (0,.5) -- (1,.5);
  \draw[very thick] (.5,.5) -- (.5,0);
  \draw[dashed] (.5,.5) -- (0,0);
  \node () at (.5,.75) {$T$};
  \end{scope}

  \begin{scope}[shift={(6,3)}]
  \fill[black!15] (0,0) rectangle (1,1);
  \node () at (1/3,3/4) {$C_m$};
  \node () at (2/3,1/4) {$C_n$};
  \draw (0,0) rectangle (1,1);
  \end{scope}

  \begin{scope}[shift={(7.5,3)}]
  \fill[black!30] (0,0) -- (.5,.5) -- (.5,0);
  \fill[black!15] (0,0) -- (.5,.5) -- (0,.5);
  \draw (0,0) rectangle (1,1);
  \draw[dashed] (.5,.5) -- (0,0);
  \draw[densely dotted] (.5,1) -- (.5,.5);
  \draw[very thick] (0,.5) -- (.5,.5) -- (.5,0);
  \node () at (.75,.5) {$R$};
  \node () at (.25,.75) {$T$};
  \end{scope}

  \begin{scope}[shift={(9,3)}]
  \draw (0,0) rectangle (1,1);
  \draw[densely dotted] (.5,0) -- (.5,1);
  \node () at (.25,.5) {$T$};
  \node () at (.75,.5) {$R$};
  \end{scope}

  \node (m) at (8.75,2.75) {$m \neq \infty$};
  \node (n) at (4.25,2.75) {$n \neq \infty$};
  \path[draw,dashed] (m) -- (10.25,2.75) -- (10.25,1.25) -- (7.25,1.25) -- (7.25,2.75) -- (m);
  \path[draw,dashed] (n) -- (5.75,2.75) -- (5.75,4.25) -- (2.75,4.25) -- (2.75,2.75) -- (n);

\end{tikzpicture}
\end{center}
\caption{The alphabet of the grid SFT $Y$ in the proof of Lemma~\ref{lem:GridsExist}.}
\label{fig:GridAlph}
\end{figure}
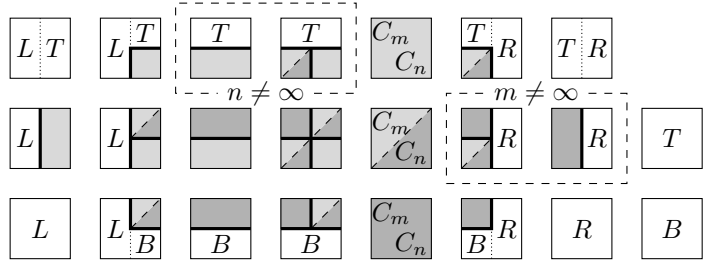

We are now ready to prove Theorem~\ref{thm:MainThm}. In its proof, we will modify the sofic shifts $X_{G(m,n)}$ by superimposing new symbols on top of their configurations with Lemma~\ref{lem:CountableSofics}. They provide a rigid geometric structure for the construction.

\begin{proof}[Proof of Theorem~\ref{thm:MainThm}]
$(1 \Rightarrow 2)$: Let $Y \subset S^{\Z^2}$ be finite, and denote $Y = \{y^1, \ldots, y^k\}$. The $y^i$ must all be periodic, and we let $p \in \N$ be a common horizontal and vertical period for all of them. Let $Z$ be a two-dimensional SFT and $\phi : Z \to X$ a surjective block map, and define the SFT $Z' \subset Z^k \times Y^k \times \{\hole,\#\}^{\Z^2}$ as follows. A configuration $z' = (z^1, \ldots, z^k, y^{n_1}, \ldots, y^{n_k}, t) \in Z^k \times Y^k \times \{\hole,\#\}^{\Z^2}$ is in $Z'$ if and only if
\begin{itemize}
\item $\{y^{n_1}, \ldots, y^{n_k}\} = \{y^1, \ldots, y^k\}$, which can be checked by $p \times p$ patterns,
\item if $t_{\vec v} = \hole$, then $\phi(z^i)_{\vec v} = y^{n_i}_{\vec v}$ for all $i \in [1,k]$, and
\item if $t_{\vec v} = \#$, then $\phi(z^i)_{\vec v} = \phi(z^j)_{\vec v}$ for all $i, j \in [1,k]$.
\end{itemize}
We then define the block map $\psi : Z' \to \hat S$ by
\[ \psi(z')_{\vec v} = \left\{ \begin{array}{ll}
	\hole, & \mbox{if~} t_{\vec v} = \hole, \\
	\phi(z^1)_{\vec v}, & \mbox{otherwise.}
\end{array} \right. \]
It is easily verified that $\psi(Z') = A(X,Y)$, and thus the extension is sofic.

$(2 \Rightarrow 3)$: Trivial.

$(3 \Rightarrow 1)$: Let $Y \subset S^{\Z^2}$ be infinite. Then for all $p \in \N$ there exists $y \in Y$ which is not $p$-periodic either in the horizontal or the vertical direction. If either condition cannot be satisfied for some $p$, then the other can be satisfied for all $p$. Thus we may assume that $Y$ has no common horizontal period. We also have $|S| \geq 2$. 

Our goal is now to construct a countably covered sofic shift $X \subset S^{\Z^2}$ such that $A(X, Y)$ is not sofic. We proceed by constructing a countably covered sofic shift $Z$ and a block map $\phi : Z \to S^{\Z^2}$ whose image we define as $X$. After this, it will be easy to show that the extension is not sofic, using Lemma~\ref{lem:Contexts}. The subshift $Z$ is a subset of $Z_1 \times Z_2 \times Z_3$, where $Z_1$ is the \emph{input layer}, $Z_2$ the \emph{computation layer}, and $Z_3$ the \emph{output layer}. We define the layers sequentially:
\begin{enumerate}
\item Define $Z_1$, a countably covered sofic shift.
\item Define the SFT $Z_2$, and a countable sofic shift $Z' \subset Z_1 \times Z_2$.
\item Define $Z_3$, a countably covered sofic shift, and the sofic shift $Z \subset Z' \times Z_3$.
\end{enumerate}


We begin with $Z_1$, which is defined by superimposing a label from $S \times \{0,1\}$ on each vertical column of the grid shift $X_{G(1,\infty)}$ given by Lemma~\ref{lem:GridsExist}. More formally, $Z_1 = Z'_1 \cap (X_{G(1,\infty)} \times (\{\square\} \cup S \times \{0,1\})^{\Z^2})$, where $Z'_1$ is the SFT with forbidden patterns $(t,\ell)$ where exactly one of $t \in \{(1,0), (1,1)\}$ and $\ell \neq \square$ holds, and all $1 \times 2$ patterns $\begin{smallmatrix} (t_1,\ell_1) \\ (t_2,\ell_2) \end{smallmatrix}$ where $\ell_1$ and $\ell_2$ differ from $\square$ and each other. The symbol $\square$ means `no label'. 
As all configurations of $X_{G(1,\infty)}$ contain only finitely many columns of $1$'s, $Z_1$ is countable, and thus a countably covered sofic shift by Lemma~\ref{lem:CountableSofics}.


We move on to the computation layer $Z_2$, whose design is very similar to that of Construction~1 in~\cite{SaTo13Pub}. Let ${*} \notin S \times \{0,1\}$ be a new symbol, and let $M = (k, k', (S \times \{0,1\}) \cup \{{*}\}, \Sigma, \delta, q_0, q_a, q_r)$ be a CMS whose functionality we define later; for now, we only require that its input alphabet is $\tilde S = S \times \{0,1\} \cup \{{*}\}$. The alphabet of $Z_2$ is $S_2 = \{ L, R, H \} \cup \left((\{0,1\} \cup \Gamma) \times \{0, 1\}^k \right)$, where $\Gamma = \delta \times (\tilde S \cup \{\#\}) \times \{{\Leftarrow}, {\leftarrow}, {\rightarrow\}}$.

First, every $2 \times 2$ pattern that contains $L$, $R$ or $H$, but is not of one of the forms
\[ \begin{array}{cccccccc}
\begin{smallmatrix} L & L \\ L & L \end{smallmatrix} &
\begin{smallmatrix} L & R \\ L & R \end{smallmatrix} &
\begin{smallmatrix} L & a \\ L & R \end{smallmatrix} &
\begin{smallmatrix} L & a \\ L & b \end{smallmatrix} &
\begin{smallmatrix} a & R \\ L & R \end{smallmatrix} &
\begin{smallmatrix} L & H \\ L & H \end{smallmatrix} &
\begin{smallmatrix} L & H \\ L & a \end{smallmatrix} &
\begin{smallmatrix} H & H \\ H & H \end{smallmatrix} \\[\medskipamount]

\begin{smallmatrix} R & R \\ R & R \end{smallmatrix} &
\begin{smallmatrix} a & R \\ R & R \end{smallmatrix} &
\begin{smallmatrix} a & b \\ c & R \end{smallmatrix} &
\begin{smallmatrix} H & R \\ R & R \end{smallmatrix} &
\begin{smallmatrix} H & H \\ H & R \end{smallmatrix} &
\begin{smallmatrix} H & H \\ a & R \end{smallmatrix} &
\begin{smallmatrix} H & H \\ a & b \end{smallmatrix} &
\end{array} \]
where $a, b, c \in (\{0,1\} \cup \Gamma) \times \{0, 1\}^k$, is forbidden. If a configuration $z \in Z_2$ contains only letters of $\{ L, R, H \}$, then it consists of two half planes, or only contains one letter. The $L$-region of $z$, if nonempty, is a left half plane, while the $R$-region is the intersection of right and southeast half planes, which may be the whole of $\Z^2$ or empty. The $H$-region, if nonempty, is the intersection of an upper half plane and the region without $L$-tiles or $R$-tiles. The rest of $z$ is called the \emph{computation cone}.

We enforce by $2 \times 1$ forbidden patterns that each horizontal row of the computation cone is of the form $1^\ell \gamma 0^{m - \ell - 1} \times \prod_{i=1}^k 1^{n_i} 0^{m - n_i}$, where $\gamma \in \Gamma$ is called the \emph{zig zag head} and $m \in \N$ is the length of the row. The product denotes a family of layers, not concatenation. The row corresponds to a computation step of $M$ where the $k$ counters have the respective values $n_1, \ldots, n_k$. The zig zag head remembers the transition of the step and the symbol read by the last counter. The length-1 row on the bottom of the cone is $(t_0, {*}, {\rightarrow}) \times \prod_{i=1}^k 0$, where $t_0 \in \delta$ is the transition $M$ takes from its initial ID $(q_0, w, 0, \ldots, 0)$ for an input word $w$ with $w_0 = {*}$.

A computation of $M$ is simulated on the rows of the cone, with time increasing upwards. The zig zag head starts from the left border of the cone in some state $(t, s, \rightarrow) \in \Gamma$, moving right at speed 2. When it hits the right border, it changes its state to $(t, s, \Leftarrow)$ and starts moving left at speed 1. If the transition $t$ involves updating a counter $i$, and the head steps on said counter, its value $n_i$ is updated. The head assumes the state $(t, s, \leftarrow)$, or $(t, r, \leftarrow)$ in case of the last counter, where $r \in \tilde S \cup \{\#\}$ is arbitrary. When the head reaches the left border in state $(t, s, \leftarrow)$, it chooses a new transition to execute, based on $t$, $s$ (interpreted as the letter under the last counter), and the set of counters with value $0$. Finally, the computation cone ends in a horizontal row of $H$-symbols precisely when the final state $q_f$ is reached. This simulation can be enforced by $2 \times 2$ forbidden patterns. The definition of $Z_2$ is complete. See Figure~\ref{fig:ComputationPic} for a visualization of the computation cone.


\input{ComputationPic2}

The subshift $Z' \subset Z_1 \times Z_2$ is defined by finitely many additional forbidden patterns. First, the width-1 bottom row of the computation cone in $Z_2$ can only be paired with the bottom left corner of the grid in $Z_1$, or the upper right corner of a pattern $\begin{smallmatrix} a & b \\ & c \end{smallmatrix}$ with $a = c = (0,0) \neq b$, and vice versa. Second, the symbol of $\tilde S \cup \{\#\}$ the zig zag head chooses when updating the last counter is determined by the symbol on the first layer at that coordinate. If this symbol has label $\ell \in S \times \{0, 1\}$, then the head gets $\ell$; if the symbol is $\#$, the head gets $\#$; otherwise, ${*}$. The upward infinite grid of $Z_1$ thus encodes an input string for the simulation of $M$.

\begin{claim}
If $M$ is deterministic, then $Z'$ is a countably covered sofic shift. Also, if a configuration $z \in Z'$ contains the bottom row of a grid of width $(n-1) n + 1$ whose labels encode a word $w \in (S \times \{0,1\})^n$, and if $M(w_0 {*}^{n-1} \cdots {*}^{n-1} w_{n-1}) = (n_1, \ldots, n_{k'}) \in \N^{k'}$, then $z$ contains a finite computation cone whose top row has the output counter values $n_1, \ldots, n_{k'}$.
\end{claim}

\begin{proof}[Proof of Claim]
For the first claim, it suffices to show that $Z'$ is countable, as the rest follows from Lemma~\ref{lem:CountableSofics}. For every $z \in Z_1$, we show that the number of $z' \in Z_2$ such that $(z, z') \in Z'$ is countable. If $z$ does not contain the bottom left corner of a grid, then $z'$ does not contain  the bottom row of a computation cone. We then have countably many choices for the positions of the different regions of $z'$, the positions of the counters, and the position and state of the zig zag head, which can make at most one infinite back-and-forth sweep. If $z$ does contain the bottom left corner of a grid, then $z'$ is completely determined, since the bottom row of the computation cone must be placed on this tile, and the simulation of $M$ is deterministic.

The latter claim follows from the fact that each sweep of the zig zag head simulates one computation step of $M$, and a simple induction argument.
\end{proof}


We now specify the CMS $M$, and for that, let $g : \N \to \N$ be any computable function with $m \geq \binom{g(m) + 2}{g(m)}$ for all $m \geq 3$, $g(m) \rightarrow \infty$ and $g(m)/m \rightarrow 0$, for example $m \mapsto \lfloor \sqrt{m} \rfloor - 1$. The machine has four output counters, and on an input word $w$, it behaves as follows. First, it checks that $w = u_0 {*}^{n-1} \cdots {*}^{n-1} u_{n-1}$ for some $n \in \N$ and $u = (u^{(1)}, u^{(2)}) \in (S \times \{0, 1\})^n$, rejecting if not. It then checks that $n = m^2 + 2r_{m^2}$ for some $m \in \N$, again rejecting if not, where $r_{m^2} = m^{2|S|^{2m^2+1}}$. Using the radius $r_{m^2}$ for the $m^2$-marker map $f_{m^2}$, it then computes the image $v = F_{m^2}(u^{(1)}) \in \{0,1\}^{m^2}$ of $u^{(1)}$ under the local function of $f_{m^2}$. By the properties of $f_{m^2}$ given in Lemma~\ref{lem:Marker}, $v$ contains at most one $1$-symbol.

Next, $M$ checks whether $v_i = 1$ for some (then necessarily unique) $i \in [0, m^2-1]$. If this is the case, then denoting $c = u^{(2)}_{r_{m^2} + i} \in \{0,1\}$, the machine $M$ outputs the four numbers $m g(m) + 1$, $g(m)$, $a g(m) + c$ and $b g(m)$, where $a, b \in [0, m-1]$ are such that $i = a + m b$. Otherwise, $M$ outputs $m g(m) + 1$, $g(m)$, $0$ and $0$. 


We then define the output layer $Z_3$, which is similar in structure to $Z_1$. It is defined by superimposing new symbols on $X_{G(2,2)}$ as follows:
\begin{enumerate}
\item Each vertical column of $1$'s gets a label from $\{\square, 0, 1\}$ and each horizontal row one from $\{\square, \$\}$. A row or column whose label is not $\square$ is called \emph{special}.
\item Every intersection of a row and a column (that is, every symbol $(1,1)$) gets a label from $\{0, 1\}$. These labels are called the \emph{elements} of the grid that contains them. The intersection of a special row and a special column has the same label as the column, and such elements are called \emph{marked}.
\item Every coordinate in a grid (every symbol except $\#$) gets a label from $\{ \tikz[scale=.4,baseline=1]{\draw[fill=black!15] (0,0) rectangle (1,1);}, \tikz[scale=.4,baseline=1]{\fill[black!15] (1,0) -- (0,1) -- (0,0); \draw (0,0) rectangle (1,1); \draw[dashed] (1,0) -- (0,1);}, \tikz[scale=.4,baseline=1]{\draw (0,0) rectangle (1,1);}\}$, with the obvious $2 \times 2$ forbidden patterns. The leftmost coordinate of a special horizontal row must have label $\tikz[scale=.4,baseline=1]{\fill[black!15] (1,0) -- (0,1) -- (0,0); \draw (0,0) rectangle (1,1); \draw[dashed] (1,0) -- (0,1);}$, and vice versa.
\end{enumerate}
Because of the diagonal signal of the third item, if a configuration of $Z_3$ contains a finite grid, then at most one of its rows can be special. Since every configuration of $X_{G(2,2)}$ contains only finitely many rows and columns of $1$'s, $Z_3$ is countable, and thus a countably covered sofic shift by Lemma~\ref{lem:CountableSofics}.

We define the subshift $Z \subset Z' \times Z_3$. First, every symbol $(a,b)$, where the $Z_2$-layer of $a$ is not $H$ and $b$ is not $\#$, is forbidden, so the grid of $Z_3$ lies on the $H$-region of $Z_2$. Recall that the four output values of the simulated CMS $M$ are of the form $m g(m) + 1$, $g(m)$, $a g(m) + c$ and $b g(m)$, where $a, b \in [0, m-1]$ and $c \in \{0, 1\}$. For a configuration $z \in Z$ and $(i,j) \in \Z^2$, we require that $z_{(i,j)}$ is on the bottom row of the $Z_3$-grid if and only if $z_{(i,j-1)}$ contains a $1$ on the sublayer of the first counter of $M$. Thus, the output $m g(m) + 1$ of the first counter is exactly the width of the $Z_3$-grid, if either (and thus both) exists. Similarly, we force the second output value $g(m)$ to be exactly the width of a grid square, by stating that if $z_{(i,j)}$ contains the bottom-most $1$ of a column and $z_{(i-1,j)}$ is also in the grid, then $z_{(i-1,j-1)} z_{(i,j-1)}$ contains $1 0$ or $0 0$ on the sublayer of the second counter. A column whose bottom coordinate is at $(i,j)$ is special if and only if $z_{(i-1,j-1)} z_{(i,j-1)} z_{(i+1,j-1)}$ contains either $1 0 0$ or $1 1 0$ on the sublayer of the third counter, the middle symbol being the label of the column. Thus the $a$'th column from the left is special, with label $c$. Finally, the position of $\tikz[scale=.4,baseline=1]{\fill[black!15] (1,0) -- (0,1) -- (0,0); \draw (0,0) rectangle (1,1); \draw[dashed] (1,0) -- (0,1);}$ on the bottom row of the grid must coincide with the fourth output counter, so the $b$'th row from the bottom is special.

See Figure~\ref{fig:AllLayers} for a visualization of a configuration of $Z$. Since the position of the output grid is determined by the computation layer, $Z$ is countable, and thus a countably covered sofic shift by Lemma~\ref{lem:CountableSofics} and the fact that countably covered sofic shifts are closed under direct product. The intuition for $Z$ is the following. A finite $Z_1$-grid (which contains $m^2 + 2r_{m^2}$ columns for some $m \in \N$) and $Z_3$-grid (which contains $m$ rows and columns with distance $g(m)$) encode two sets $A, B \subset [0,m-1]^2$ in a configuration that contains both of them. The set $A$ is given by the binary word $v_{[r_{m^2}, r_{m^2} + m^2 - 1]}$ of length $m^2$, where $v \in \{0,1\}^{m^2 + 2r_{m^2}}$ consists of the binary labels of the $Z_1$-columns, while $B$ is encoded in the elements of the $Z_3$-grid. We can `mark' one of the elements of $[0,m-1]^2$ by selecting the $S$-labels of the $Z_1$-columns so that the marker map computed by $M$ places a $1$ at that exact coordinate. The output counters of $M$ force the corresponding element of the $Z_3$-grid to be marked, and then $A$ and $B$ must both contain or both lack that element.

\begin{figure}
\begin{center}
\begin{tikzpicture}[scale=1.1]

\fill[black!15] (1.125,8) -- (0,8) -- (0,9.125);
\draw[dotted] (1.125,8) -- (0,9.125);

\draw[thick] (0,2) -- (0,10);
\draw[thick] (0,2) -- (4,6);
\draw[thick] (0,8) -- (4,8);
\draw[thick] (0,9.5) -- (1.5,9.5) -- (1.5,8);
\draw[thick] (2,2) -- (2,4);

\draw[thick] (0,2) -- (2,2);

\foreach \x in {0.5,1,1.5}{
	\draw[densely dotted] (\x,2) -- (\x,8);
}
\draw[densely dotted] (2,4) -- (2,8);

\draw[dashed]
(0,2) -- (.25,2.5) -- (.75,3) -- (.75,3.5) --
(.75,4) -- (1,4.5) -- (1.5,5) -- (1.85,5.5) --
(2,6) -- (2.25,6.5) -- (2,7) -- (1.5,7.5) -- (1.125,8);

\draw[dashed]
(0,3.5) -- (.25,4) -- (.4,4.5) -- (.4,5) -- (.4,5.5) --
(.75,6) -- (.75,6.5) -- (1.25,7) -- (1.25,7.5) -- (1.5,8);

\draw[dashed]
(0,4) -- (.3,4.5) -- (.6,5) -- (.6,5.5) --
(.6,6) -- (.3,6.5) -- (.6,7) -- (.75,7.5) -- (.75,8);

\draw[dashed]
(0,6.5) -- (.2,7) -- (.375,7.5) -- (.375,8);


\foreach \a in {.375,.75}{
	\draw (0,\a+8) -- (1.5,\a+8);
}
\draw[densely dotted] (0,9.125) -- (1.5,9.125);
\foreach \a in {.375,1.125}{
	\draw (\a,8) -- (\a,9.5);
}
\draw[densely dotted] (.75,8) -- (.75,9.5);

\foreach \x in {0,.375,.75,1.125,1.5}{
	\foreach \y in {0,.375,.75,1.125,1.5}{
		\draw[fill=white] (\x,\y+8) circle (.05);
	}
}

\fill (.75,9.125) circle (.05);


\draw[decorate,decoration={zigzag,amplitude=1.5,segment length=7.1}] (4,1.7) -- (4,10);

\foreach \x in {0,0.5,1,1.5,2}{
	\draw[fill=white] (\x+5,1.95) circle (.05);
	\draw[fill=black!20] (\x+5,1.85) circle (.05);
	\draw[fill=black] (\x+5,1.75) circle (.05);
}

\draw[fill=black] (4.95,7.85) rectangle (6.55,7.925);


\foreach \x in {0,.375,.75,1.125,1.5}{
	\foreach \y in {0,.375,.75,1.125,1.5}{
		\draw[fill=white] (\x+5,\y+8) circle (.05);
	}
}

\end{tikzpicture}
\end{center}
\caption{Left: a schematic diagram of a configuration of $Z$ containing a finite $Z_3$-grid, not drawn to scale. The dotted lines represent the columns of the $Z_1$-grid, whose labels are read by the CMS running inside the triangular computation cone. The dashed lines represent the output counters. The computation ends at the base of the $Z_3$-grid, whose elements are denoted by small circles, of which the filled one is marked. Right: its $\Phi$-image. The white, gray and black circles represent symbols picked from $\{0,1\}$, $S$ and $\{1\}$, respectively. The horizontal bar is a row of symbols 1.}
\label{fig:AllLayers}
\end{figure}
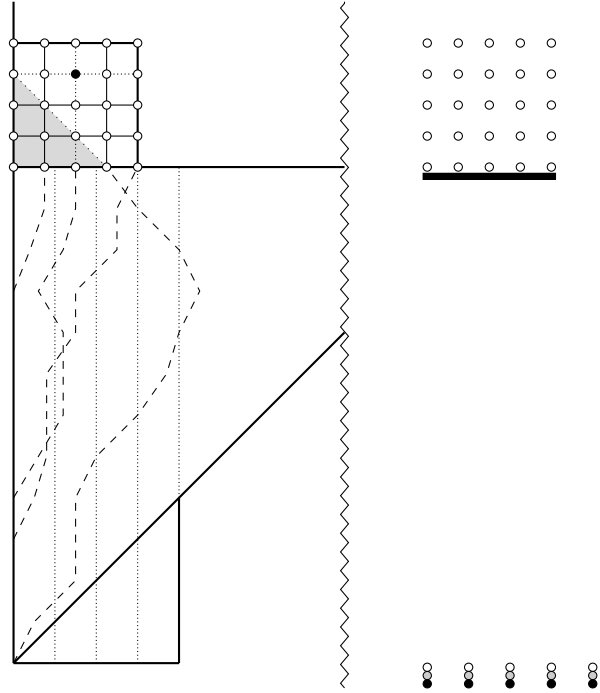


Finally, define the block map $\phi : Z \to S^{\Z^2}$ with neighborhood $\{0\} \times \{0,1,2,3\} \subset \Z^2$ and local function $\Phi$ as follows. Without loss of generality, assume $0, 1 \in S$. Let $c_1, c_2, c_3, c_4$ be symbols of the alphabet of $Z$. First, if the $Z_1$-layer of $c_4$ is $\#$, then
\[
\Phi \left(\begin{smallmatrix} c_1 \\ c_2 \\ c_3 \\ c_4 \end{smallmatrix}\right) = \left\{ \begin{array}{rl}
b, & \mbox{if the $Z_1$-layer of $c_3$ has label $(a,b) \in S \times \{0,1\}$,} \\
a, & \mbox{if the $Z_1$-layer of $c_3$ is $\#$ and that of $c_2$ has label $(a,b)$,} \\
1, & \mbox{if the $Z_1$-layer of $c_2$ is $\#$ and that of $c_1$ is not,} \\
0, & \mbox{otherwise.}
\end{array} \right.
\]
Second, if $c_4$ is an element of the $Z_3$-grid, then the image of $\Phi$ at that coordinate is its label (in $\{0, 1\}$). Finally, if the $Z_2$-label of $c_4$ and the $Z_3$-label of $c_3$ are not $\#$, then the $\Phi$-image is $1$. Everything else is mapped to $0$.

Let $m \geq 3$, $u \in S^{m^2 + r_{m^2}}$, $v \in \{0,1\}^{m^2 + r_{m^2}}$ and $B \subset [0, m-1]^2$ be such that if the image $F_{m^2}(u)$ of the marker map contains a $1$ at some (necessarily unique) coordinate $i = a + b m$, then $v_{r_{m^2} + i} = 1$ if and only if $(a,b) \in B$. Then there exists $x = x^{m,u,v,B} \in X$ as follows. For every $(i,j) \in B$, we have $x_{(g(m) i, g(m) j)} = 1$ (given by the elements of the $Z_3$-grid). We also have $x_{(i,-1)} = 1$ for all $i \in [0, m g(m)]$. For some $n < 0$, $x$ has a horizontal row of $1$'s of length $m(m^2 + 2 r_{m^2}) + 1$ to the right of $(0,n) \in \Z^2$, above which are the two rows $u_0 0^{m-1} u_1 0^{m-1} \cdots 0^{m-1} u_{m^2 + 2 r_{m^2} - 1}$ and $v_0 0^{m-1} v_1 0^{m-1} \cdots 0^{m-1} v_{m^2 + 2 r_{m^2} - 1}$ of the same length (given by the bottom rows of the $Z_1$-grid). For all other $\vec v \in \Z^2$, we have $x_{\vec v}$ = 0. Configurations of $X$ which are not translates of some $x^{m,u,v,B}$ do not contain such finite rows of $1$'s.



We now show that the extension $A(X,Y) \in \hat S^{\Z^2}$ is not sofic. Let $x = x^{m,u,v,B} \in X$ and $n < 0$ be as above, and denote by $A \subset [0,m-1]^2$ the set encoded by $v_{[r_{m^2}, r_{m^2} + m^2 - 1]}$. Construct a new configuration $\hat x \in \hat S^{\Z^2}$ by replacing those symbols in $x$ that encode $u$ with $\hole$-symbols. We claim that $\hat x \in A(X,Y)$ holds if and only if $A = B$. Suppose first $A = B$, and let $y \in Y$ be arbitrary. If the $F_{m^2}$-image of the word $y_{(0,1-n)} y_{(m,1-n)} \cdots y_{(m(m^2 + 2 r_{m^2} - 1),1-n)} \in S^{m^2 + 2 r_{m^2}}$ substituted to the $\hole$-symbols of $\hat x$ contains a $1$, then the respective coordinate of $[0,m-1]^2$ is in $A$ if and only if it is in $B$, and we have $\hat x^{(y)} \in X$; otherwise, this holds automatically.

Conversely, suppose we have $A \neq B$, so there exists $i = a + m b \in [0,m^2-1]$ such that $v_{r_{m^2} + i} = 1$ but $(a,b) \notin B$, or vice versa. Denote $k = m^2 + 2r_{m^2}$ and $D = \{ (0,1), (k,1), \ldots, (k(k-1),1) \} \subset \Z^2$. Since $Y$ is not horizontally periodic, there exists a pattern $P \in \B_D(Y)$ that, when interpreted as a word $u \in S^k$, satisfies $F_{m^2}(u)_i = 1$ for the marker map $F_{m^2}$. We now have $\hat x^{(y)} \notin X$, as otherwise the machine $M$ would mark the column $a$ and row $b$ in its $\phi$-preimage, and we would have $(a,b) \in B$ because $(a,b) \in A$, a contradiction. Thus $\hat x \notin A(X,Y)$, and we have shown that $\hat x \in A(X,Y)$ if and only if $A = B$.

Suppose for a contradiction that $A(X,Y)$ is sofic, and let $C > 0$ be given for it by Lemma~\ref{lem:Contexts}. Let $B \subset [0,m-1]^2$, and let $\hat x^{m, B} = \hat x^{m,u,v,B}$, where $u \in S^{m^2 + r_{m^2}}$ is arbitrary, and $v \in \{0,1\}^{m^2 + r_{m^2}}$ encodes the set $B$. There are $2^{m^2}$ such configurations for a given $m \in \N$, and $C^{mg(m)} < 2^{m^2}$ holds when $m$ is large enough. Then there are $B \neq B' \subset [0,m-1]^2$ such that, with the notation of Lemma~\ref{lem:Contexts}, we have $c(\hat x^{m,B}, \hat x^{m,B'}, mg(m)+1) \in A(X,Y)$. But this configuration is exactly $\hat x^{m,u,v',B}$ where $v' \in \{0,1\}^{m^2 + r_{m^2}}$ encodes the set $B'$, a contradiction since $B \neq B'$.
\end{proof}

\section{Further Discussion}

In this section, we discuss some variants and strengthenings of Theorem~\ref{thm:MainThm}, and present a few open problems. We begin by defining the notion of downward determinism for two-dimensional subshifts.

\begin{definition}
Denote by $H$ the upper half-plane $\{ (a,b) \in \Z^2 \;|\; b > 0 \}$. A two-dimensional subshift $X \subset S^{\Z^2}$ is \emph{downward deterministic} if $x_H = y_H$ for two configurations $x, y \in X$ implies $x_{\vec 0} = y_{\vec 0}$. By compactness, this implies the existence of $n \in \N$ such that already $x_{[-n,n] \times [1,n]} = y_{[-n,n] \times [1,n]}$ implies $x_{\vec 0} = y_{\vec 0}$.
\end{definition}

The subshift $X$ constructed in the proof of Theorem~\ref{thm:MainThm} is not downward deterministic, which leads us to studying the extensions of downward deterministic subshifts. The next result may seem surprising, but the proof is elementary.

\begin{proposition}
Let $X \subset S^{\Z^2}$ be a downward deterministic subshift, and let $Y \subset S^{\Z^2}$ be any subshift. If $A(X,Y) \neq X$, then $Y$ is downward deterministic, and if $Y$ contains at least $2$ points, then $A(X,Y)$ is downward deterministic.
\end{proposition}


\begin{proof}
Let $n \in \N$ be such that the rectangle $x_{[-n,n] \times [1,n]}$ determines $x_{\vec 0}$ for all $x \in X$, and denote $D = [-n,n] \times [1,n] \cup \{\vec 0\}$. Suppose first that $A(X,Y) \neq X$, and let $P \in \B_D(A(X,Y))$ be such that $P_{\vec 0} = \hole$. Let $Q, R \in \B_D(Y)$ be such that $Q_{[-n,n] \times [1,n]} = R_{[-n,n] \times [1,n]}$, and consider the substitutions $P^{(Q)}, P^{(R)} \in \B_D(X)$. We have $P^{(Q)}_{[-n,n] \times [1,n]} = P^{(R)}_{[-n,n] \times [1,n]}$, which implies  $P^{(Q)}_{\vec 0} = P^{(R)}_{\vec 0}$, and thus $Q_{\vec 0} = R_{\vec 0}$. But this means that $Y$ is downward deterministic.

Suppose then that $Y$ is nontrivial, and let $P, Q \in \B_D(A(X,Y))$ be such that $P_{[-n,n] \times [1,n]} = Q_{[-n,n] \times [1,n]}$. Now, if $P_{\vec 0} \in S$, let $R \in \B_D(Y)$ be such that $R_{\vec 0} \neq P_{\vec 0}$. Such an $R$ exists, since at least two distinct letters occur in $Y$. Since $P^{(R)}, Q^{(R)} \in \B(X)$ agree on the set $[-n,n] \times [1,n]$, we must have $Q_{\vec 0} = P_{\vec 0}$. Symmetrically, if $Q_{\vec 0} \in S$, then $P_{\vec 0} = Q_{\vec 0}$, and thus $P_{\vec 0} = \hole$ if and only if $Q_{\vec 0} = \hole$. Thus $A(X,Y)$ is downward deterministic.
\end{proof}

The subshifts constructed in Proposition~\ref{prop:CompuNonsofic} are vertically constant, thus downward deterministic, so even for downward deterministic sofic shifts $X, Y \subset S^{\Z^2}$, the universal extension $A(X,Y)$ need not be sofic. However, the construction relies on $Y$ being computationally difficult, and not much can be said if $Y$ is recursive. Namely, by the previous result, if there are downward deterministic sofic shifts $X$ and $Y$, with $Y$ recursive, such that $A(X,Y)$ is not sofic, then $A(X,Y)$ is a downward deterministic $\Pi^0_1$ subshift which is not sofic, and it is currently unknown whether such an object exists. In particular, Lemma~\ref{lem:Contexts} cannot be applied, since all downward deterministic subshifts also satisfy its conclusion. In \cite{De06}, it was proved that every multidimensional sofic shift $X$ has a collection of subsystems whose entropies are dense in the interval $[0, h(X)]$. But since downward deterministic subshifts have zero entropy, they always satisfy this condition, too.

\begin{problem}
Let $X, Y \subset S^{\Z^2}$ be downward deterministic sofic shifts, and let $\B(Y)$ be recursive. Is $A(X,Y)$ necessarily sofic?
\end{problem}

Theorem~\ref{thm:MainThm} also has the following dual problem.

\begin{problem}
For a given sofic shift $X \subset S^{\Z^2}$, does there exist a (sofic/recursive) subshift $Y \subset S^{\Z^d}$ such that $A(X,Y)$ is not sofic?
\end{problem}

By Proposition~\ref{prop:SFTs}, the answer is negative in both cases if $X$ is an SFT. Proposition~\ref{prop:CompuNonsofic} and Theorem~\ref{thm:MainThm} show that there exist some particular and quite intricate $X$ for which the answer is positive in both cases.

Finally, recall from Proposition~\ref{prop:Soficness} that the existential extension of every one-dimensional SFT by another SFT is sofic, but Example~\ref{ex:1DSFT} showed that it may not be an SFT, even when extending by a full shift. This raises the following question.

\begin{problem}
Let $X, Y \subset S^\Z$ be SFTs. When is $E(X,Y)$ an SFT?
\end{problem}

\section*{Acknowledgments}

I am thankful to Brian Marcus, Michael Schraudner and Ronnie Pavlov for reading through several early versions of this paper, Ville Salo for lengthy discussions on the subject, and the anonymous referee for their valuable comments that greatly helped to improve the readability of this paper.

\bibliographystyle{amsplain}
\bibliography{../../../bib/bib}{}

\end{document}